\documentclass[11pt, reqno]{amsart}
\usepackage{amsmath}%
\usepackage{amssymb}%
\usepackage[english]{babel}
\usepackage[utf8]{inputenc}
\usepackage{hyperref}
\usepackage{amsfonts}%
\usepackage{setspace}
\usepackage{enumerate}
\usepackage{graphicx}
\usepackage{mathrsfs}
\usepackage{xcolor}

\usepackage[T1]{fontenc}
\usepackage{appendix}
    \numberwithin{equation}{section}

\newcommand{\diag}{{diag}}
\newcommand{\spn}{{span}}
\newcommand{\F}{\mathbb{F}}
\newcommand{\R}{\mathbb{R}}
\newcommand{\bP}{\mathbb{P}}
\newcommand{\Hq}{\mathbb{H}}

\newcommand{\C}{\mathbb{C}}
\newcommand{\M}{\mathcal{M}}
\newcommand{\Conv}{{\rm conv}}

\newcommand{\D}{\mathbb{D}}

\newcommand{\bS}{\mathbb{S}}
\newcommand{\vc}[1]{\boldsymbol{#1}}
\newcommand{\ch}{\text{conv}}

\newcommand\xqed[1]{%
  \leavevmode\unskip\penalty9999 \hbox{}\nobreak\hfill
  \quad\hbox{#1}}

\newcommand\squareend{\xqed{$\square$}}

\newtheorem{theorem}{Theorem}[section]

 \newtheorem{corollary}[theorem]{Corollary}
 
 \newtheorem{lemma}[theorem]{Lemma}
 \newtheorem{proposition}[theorem]{Proposition}
 \theoremstyle{definition}

 \newtheorem{example}[theorem]{Example}
 \numberwithin{equation}{section}

\makeatletter
\newtheorem*{rep@theorem}{\rep@title}
\newcommand{\newreptheorem}[2]{
\newenvironment{rep#1}[1]{%
\def\rep@title{#2 \ref{##1}}%
\begin{rep@theorem}}%
{\end{rep@theorem}}}
\makeatother

\newreptheorem{corollary}{Corollary}

\begin{document}

\thanks{The second author was partially supported by FCT through project UID/MAT/04459/2019 and the third author was partially supported by FCT through CMA-UBI, project PEst-OE/MAT/UI0212/2020.}

\title{Quaternionic Numerical Range of Complex Matrices}

\author[L. Carvalho]{Lu\'{\i}s Carvalho}
\address{Lu\'{\i}s Carvalho, ISCTE - Lisbon University Institute\\    Av. das For\c{c}as Armadas\\     1649-026, Lisbon\\   Portugal}
\email{luis.carvalho@iscte-iul.pt}
\author[Cristina Diogo]{Cristina Diogo}
\address{Cristina Diogo, ISCTE - Lisbon University Institute\\    Av. das For\c{c}as Armadas\\     1649-026, Lisbon\\   Portugal\\ and \\ Center for Mathematical Analysis, Geometry,
and Dynamical Systems\\ Mathematics Department,\\
Instituto Superior T\'ecnico, Universidade de Lisboa\\  Av. Rovisco Pais, 1049-001 Lisboa,  Portugal
}
\email{cristina.diogo@iscte-iul.pt}
\author[S. Mendes]{S\'{e}rgio Mendes}
\address{S\'{e}rgio Mendes, ISCTE - Lisbon University Institute\\    Av. das For\c{c}as Armadas\\     1649-026, Lisbon\\   Portugal\\ and Centro de Matem\'{a}tica e Aplica\c{c}\~{o}es \\ Universidade da Beira Interior \\ Rua Marqu\^{e}s d'\'{A}vila e Bolama \\ 6201-001, Covilh\~{a}}
\email{sergio.mendes@iscte-iul.pt}

\date{\today}

\maketitle

\begin{abstract}
The paper explores further the computation of the quaternionic numerical range of a complex matrix. We prove a modified version of a conjecture by So and Tompson. Specifically, we show that the shape of the quaternionic numerical range for a complex matrix depends on the complex numerical range and two real values. We establish under which conditions the bild of a complex matrix coincides with its complex numerical range and when the quaternionic numerical range is convex.
\end{abstract}
\maketitle
\maketitle
\section{Introduction}
Let $\Hq$ denote the Hamilton quaternions and let $\M_n(\Hq)$ be the set of $n\times n$ matrices with quaternionic entries. The quaternionic numerical range of a given matrix $A \in \M_n(\Hq)$, denoted by $W_{\Hq}(A)$, is the set of $\vc{x}^*A\vc{x}$, with $\vc{x}$ running over the quaternionic unit sphere of $\Hq^n$. Apart from special cases, such as normal matrices \cite{STZ} and real matrices  \cite{CDM}, little is known about the computation of the quaternionic numerical range.
Such difficulty in computing $W_{\Hq}(A)$ was one of the reasons that led Kippenhahn \cite{Ki} to introduce  the bild $B(A)$ of $A$, that is, the intersection of $W_{\Hq}(A)$ with the complex plane.
 In fact, since every element of $W_{\Hq}(A)$ is similar to an element of the closure of the upper half plane, it is enough to consider the upper bild $B^+(A)=B(A)\cap\C^+$.



In \cite[theorem 7.1]{ST} it is proved that the quaternionic numerical range of a $2\times2$ complex matrix $A$ is determined by its complex numerical range.  Specifically, when $W_{\C}(A) \cap \R\neq \emptyset$ then $B^+(A) =\ch\{W_{\C}(A)\cap \C^+, W_{\C}(A^*)\cap \C^+\}$; when $ W_{\C}(A) \cap \R=\emptyset$, we have that $B^+(A) =\ch\{W_{\C}(A)\cap \C^+, T\}$, for a certain real $T$.
In the same paper \cite[10; p.364]{ST} it is conjectured that these results can be generalized for $n\times n$ complex matrices, see  \cite[10(i)-(iii); p.364]{ST}.
 However these conjectures prove to be untrue. This can be seen from the fact that there exist matrices $A=H+Si\in \M_n(\C)$, with $S$ a diagonal positive definite matrix, which are unitary equivalent to matrices $\tilde A=\tilde H+\tilde Si$, with $\tilde S$ a diagonal indefinite matrix. Since $W_{\C}(A)\cap \R=\emptyset$ and $W_{\C}(\tilde A)\cap \R\neq\emptyset$, the conjectures propose formulas for $W_\Hq(A)$ and $W_\Hq(\tilde A)$, which in some cases may differ. This is a contradiction because we know that $W_\Hq(A)=W_\Hq(\tilde A)$. Next example shows, using only normal matrices and  the result \cite[Main Theorem, p.192]{STZ}, the failure of the conjecture.
\begin{example}\label{example_intro}
Let $A\in\M_{4}(\C)$ be the normal matrix $A=\diag(-1-i,-1-i,1+i,1+i)$. The complex numerical range of $A$ is $W_{\C}(A)=[-1-i,1+i]$ and that of $A^*$ is $W_{\C}(A^*)=[-1+i,1-i]$. Since $W_{\C}(A) \cap \R=\{0\}$, the conjectured upper bild  \cite[10(ii),(iii); p.364]{ST} is
$$\ch\{W_{\C}(A) \cap \C^+, W_{\C}(A^*) \cap \C^+\}=\ch\{-1+i, 1+i,0\}.$$
%
%
However, using \cite{STZ} the upper bild is the square $B^+(A)=\ch\{-1+i, 1+i,-1, 1\}$. It is also known that the matrix $A$ is unitary equivalent (in $\Hq$) to $\tilde{A}=\diag(-1+i,-1+i,1+i,1+i)$, whose complex numerical range lies on the upper bild. If the conjecture raised in \cite {ST} was true, the upper bild would be the triangle $\ch\{-1+i,1+i, T\}$, for some real $T$, which is different from the above $B^+(A)$. \squareend
\end{example}

So and Thompson's result \cite[theorem 7.1]{ST} is interesting as it provides a way to relate the complex with the quaternionic numerical range.
It allows to use the extensive body of knowledge established for the complex numerical range and bring it to the quaternionic field.
It turns out that we can prove a slightly modified version of this theorem. Theorem \ref{upperbild} shows that, for a complex matrix $A \in\M_n(\C)$, we have $B^+(A)=\ch\{W_{\C}(A) \cap \C^+, W_{\C}(A^*) \cap \C^+, \underline v, \overline v\}$, with $\underline v, \overline v \in \R$. In this sense this is a follow up of \cite{CDM}, where we concluded that the complex numerical range and the bild coincide for real matrices.

 We start section \ref{preliminaries_section} with recalling a few results about the numerical range in the quaternionic setting and we also fix notation to be used throughout the text.
In section \ref{shape_section} we give a characterization of the bild $B(A)$ for complex matrices $A \in\M_n(\C)$ (proposition \ref{bild_expr}). This allows us to compute $B(A)$ and then infer about the quaternionic numerical range. In corollary \ref{cor_bild} we prove that $B(A)$ is always included in the convex hull of $W_{\C}(A)$ and $W_{\C}(A^*)$. Although the complex and quaternionic numerical radius coincide for complex matrices (see corollary \ref{numerical radius}), example \ref{tlvbk} shows that, contrary to the literature, the quaternionic numerical radius is not a norm in general, in a striking difference with its complex counterpart. Theorem \ref{upperbild} is the main result of the article, as it proposes a shape for the numerical range inspired in the aformentioned conjecture proposed by So and Thompson. It is then used as a stepping stone to several other results.  Corollary \ref{bild_complex_nr} gives a necessary and sufficient condition for the bild  $B(A)$ to coincide with the complex numerical range $W_{\C}(A)$. As a consequence, it provides a way to determine if the numerical range of a complex matrix is convex without requiring its computation. In corollary \ref{upperbild_2times2}, theorem  \ref{upperbild} is used to clarify why the numerical range has the shape defined in \cite{ST} for $2\times 2$ complex matrices.

%
%
%
%

%
\section{Preliminaries}\label{preliminaries_section}
The Hamiltonian quaternions $\Hq$ is an algebra over $\R$ with basis $\{1, i, j, k\}$. The product in $\Hq$ is given by $i^2=j^2=k^2=ijk=-1$. The pure quaternions are denoted by $\bP=\mathrm{span}_{\R}\,\{i,j,k\}$. The real and imaginary parts of a quaternion $q=a_0+a_1i+a_2j+a_3k\in\Hq$ are denoted by $Re(q)=a_0$ and $Im(q)=a_1i+a_2j+a_3k$, respectively. The conjugate of $q$ is given by $q^*=Re(q)-Im(q)$ and its norm is $|q|^2=qq^*$. Two quaternions $q_1,q_2\in\Hq$ are called similar, $q_1\sim q_2$, if there exists $s \in \Hq$ with $|s|=1$ such that $s^{*}q_2 s=q_1$. The equivalence class containing all the quaternions similar to $q$ is denoted by $[q]$. A necessary and sufficient condition for the similarity of $q_1$ and $q_2$ is that $Re(q_1)=Re(q_2) \textrm{ and }|Im(q_1)|=|Im(q_2)|$. The segment joining two quaternions $q_1, q_2 \in \Hq$ is denoted by  $[q_1, q_2]=\{\alpha q_1 +(1-\alpha)q_2: \alpha \in [0,1]\}$.

Let $\F$ denote $\R$, $\C$ or $\Hq$. Let $\F^n$ be the $n$-dimensional $\F$-space. For $\vc{x}\in \F^n$, $\vc{x}^*$ is the conjugate transpose of $\vc{x}$; and the unitary vector with the same direction as $\vc{x}$ is $\vc{x}_{\bS} \in \bS_{\F^n}$, thus $\vc{x}= |\vc{x}| \vc{x}_{\bS}$. The disk with centre $\vc{a}\in\F^n$ and radius $r>0$ is the set $\D_{\F^n}(\vc{a},r)=\{\vc{x}\in\F^n:|\vc{x}-\vc{a}|\leq r\}$ and its boundary $\partial \D_{\F^n}(\vc{a},r)$ is the sphere $\bS_{\F^n}(\vc{a},r)$. In particular, if $\vc{a}=0$ and $r=1$, we simply write $\D_{\F^n}$ and $\bS_{\F^n}$. With this notation, the group of unitary quaternions is denoted by $\bS_{\Hq}$.

Let $\M_{n} (\F)$ be the set of all $n\times n$ matrices with entries over $\F$. For $A\in\M_{n} (\F)$, $\bar{A}$ and $A^*$ denote the conjugate and the conjugate transpose of $A$, respectively.

The set
\[
W_{\F}(A)=\{\vc{x}^*A\vc{x}: \vc{x}\in \bS_{\F^n}\}
\]
is called the numerical range of $A$ in $\F$.  The numerical range of $A$ is invariant under unitary equivalence, that is,  $W_{\F}(A)=W_{\F}(U^*AU)$, for every unitary $U\in \M_n(\F)$ (\cite[theorem 3.5.4]{R}).

 It is well known that  $q\in W_{\Hq}(A)$ is equivalent to $[q]\subseteq W_{\Hq}(A)$ \cite[p.38]{R}. Therefore, it is enough to study the subset of complex elements in each similarity class. This set is known as $B(A)$, the bild of $A$,
\[
B(A)=W_{\Hq}(A)\cap \C.
\]
Although the bild may not be convex, the upper bild $B^+(A)=W_{\Hq}(A)\cap \C^+$ is always convex, see \cite{ST}. When $\F=\C$, we denote $W_{\C}^+(A)=W_{\C}(A)\cap\C^+$.

Let $\F$ be $\C$ or $\Hq$. The complex and the quaternionic numerical radius of $A$ is given by
\[
w_{\F}(A)=\max\{|z|: z\in W_{\F}(A)\}.
\]
Note that $\max\{|z|: z\in B(A)\}$ coincides with $w_{\Hq}(A)$.

Taking into account that $\F$ can be seen as a real subspace of $\Hq$, we denote the projection of $\Hq$ over $\F$ by $\pi_{\F}:\Hq \rightarrow \F$. The projection of $W_{\Hq}(A)$ over $\F$ is $\pi_{\F}(W_{\Hq}(A))=\{\pi_{\F}(\omega): \omega\in W_{\Hq}(A)\}$.

Given $A \in \M_n(\Hq)$ there exists an associated complex matrix
\[
\chi(A)=\left[
\begin{array}{cc}
A_1 &  A_2 \\
-\bar A_2 & \bar A_1
\end{array}
\right]\in\M_{2n}(\C),
\]
where $A_1, A_2 \in \M_n(\C)$ and $A=A_1+A_2j$. Au-Yeung found necessary and sufficient conditions for the convexity of $W_{\Hq}(A)$ in terms of the complex numerical range of $\chi(A)$ (see \cite{Ye1, Ye2}).



For any given matrix $A \in \M_n(\C)$, we will work with the usual Hermitian and skew-Hermitian complex decomposition of $A$, $A=H+Si$, assuming furthermore that the matrix $S$ is real diagonal. In fact, any complex matrix $A$ can be written as $A=\tilde H+\tilde S$, with $\tilde H =\frac{A+A^*}{2}$ Hermitian and $\tilde S =\frac{A-A^*}{2}$ skew-Hermitian. Now if we take $U \in \M_n(\C)$ a unitary which diagonalizes $\tilde S$, we obtain $Si=U^*\tilde S U$, with $S$ real diagonal. It is worth pointing out that we used complex unitary matrices, not quaternionic, and thus the matrix $U^*\tilde{H}U$ is also complex.
Since the numerical range is invariant under unitary equivalence, we can work with $U^*AU= U^*\tilde H U+U^*\tilde S U=H+Si$. Therefore, unless mentioned otherwise, we will only consider matrices of the form $A=H+Si \in \M_n(\C)$, where $H \in \M_n(\C)$ is Hermitian and $S \in \M_n(\R)$ is diagonal.
Since $q\in \bS_{\bP}$ is equivalent to $i$, $q\sim i$, then any matrix with entries only in $\spn\{1, q\}$ is unitary equivalent to complex matrices. Hence, the results also apply to such matrices $A\in\M_n(\spn\{1, q\})$.

We recall that a matrix $S \in \M_n(\R)$ is positive definite  (resp., positive semi-definite) if $\vc{x}^*S\vc{x}>0$ (resp. $\vc{x}^*S\vc{x}\geq 0$),  and negative definite  (resp., negative semi-definite) if $\vc{x}^*S\vc{x}<0$ (resp. $\vc{x}^*S\vc{x}\leq 0$), for all $\vc{x} \in \bS_{\R^n}$.
Moreover, $S$ is indefinite if there are $\vc{x}, \vc{y} \in \bS_{\R^n}$ such that $\vc{x}^*S\vc{x}>0$ and $\vc{y}^*S\vc{y}<0$.

\section{Numerical range of complex matrices} \label{shape_section}
Fundamental to understand the quaternionic numerical range of a matrix $A\in \M_{n}(\Hq)$ is its bild, in view of the similarity relation $B(A)= W_{\Hq}(A)/\sim$. The bild of $A$, however, is in general difficult to compute and the usual procedure is to obtain first the quaternionic numerical range.

When $A$ is a complex matrix, we show in the present paper that it is possible to reverse this approach, that is, to compute the bild and then infer about the quaternionic numerical range. For this matters we start by characterizing those elements $\vc{q}=\vc{x}+\vc{y}j\in \bS_{\Hq^n}$ such that $\vc{q}^*A\vc{q}\in B(A)$.

Let $\mathcal{D}_0$ be the set defined by
\[
\mathcal{D}_0=\{ (\vc{x},\vc{y})  \in \bS_{\C^{2n}}: \vc{q}^*A\vc{q}\in B(A), \textrm{ with }\vc{q}=\vc{x}+\vc{y}j\in \bS_{\Hq^n}\}.
\]
Our first result shows that the elements in the bild are of the form $\vc{x}^*A\vc{x}+\vc{y}^*A^*\vc{y}$, with $(\vc{x}, \vc{y})\in \mathcal{D}_0$.
Moreover, the condition for $(\vc{x}, \vc{y})\in \mathcal{D}_0$ is that $\vc{x}^*S\vc{y}=0$, where $A=H+Si\in \M_{n}(\C)$.

%

%
%
\begin{proposition} \label{bild_expr}
Given a complex matrix $A \in \M_n(\C)$, we have:
\begin{equation*}
\mathcal{D}_0=\{ (\vc{x},\vc{y})  \in \bS_{\C^{2n}}: \vc{x}^*S\vc{y}=0\}
\end{equation*}
and
\begin{equation}
B(A)=\Big\{\vc{x}^*A\vc{x}+\vc{y}^* A^* \vc{y}: (\vc{x}, \vc{y})\in \mathcal{D}_0 \Big\}. \label{prop_eq_1}
\end{equation}
\end{proposition}
\begin{proof}
An element $\omega$ of the numerical range of $A \in \M_n(\C)$ is of the form
\begin{align}
\omega & =\vc{q}^*A\vc{q}=(\vc{x}+\vc{y}j)^*A(\vc{x}+\vc{y}j) \nonumber\\
& =(\vc{x}^*-j\vc{y}^*)A(\vc{x}+\vc{y}j) \nonumber\\
& =\vc{x}^*A\vc{x}+\vc{x}^*A\vc{y}j - j\vc{y}^* A\vc{x}-j\vc{y}^* A\vc{y}j\nonumber\\
& =\vc{x}^*A\vc{x}+\vc{x}^*A\vc{y}j - \vc{x}^* A^*\vc{y}j+\vc{y}^* A^*\vc{y}\nonumber\\
& =\big(\vc{x}^*A\vc{x}+\vc{y}^* A^*\vc{y}\big)+\vc{x}^* \big(A - A^*\big)\vc{y}j. \label{num_range_element0}
\end{align}
Since $A-A^*$ is the skew-Hermitian $2Si$, where $S$ is a real diagonal matrix, we have:
\[
\omega =\vc{x}^*A\vc{x}+\vc{y}^* A^*\vc{y}+ 2\vc{x}^*S\vc{y}k.
\]

We see that $(\vc{x}, \vc{y})\in \mathcal{D}_0$, that is, $\omega\in B(A)$, if and only if, $\vc{x}^*S\vc{y}=0$, in which case we have
\begin{align}\label{num_range_element2}
\omega & =\vc{x}^*A\vc{x}+\vc{y}^* A^*\vc{y}.
\end{align}\end{proof}

Taking into account the decomposition $A=H+Si$, the bild of $A$ may also be written as
\begin{equation}\label{B(A) in terms of H,S}
B(A)=\Big\{\vc{x}^*H\vc{x}+\vc{y}^* H \vc{y}+(\vc{x}^*S\vc{x}-\vc{y}^*S\vc{y})i: (\vc{x},\vc{y})\in\mathcal{D}_0 \Big\}.
\end{equation}

One consequence of our first proposition is a criterium for an element $\vc{q}^*A\vc{q}$ of the quaternionic numerical range to be in the reals. We have
\begin{equation}\label{Real part of bild}
B(A)\cap\R=\Big\{\vc{x}^*H\vc{x}+\vc{y}^* H \vc{y}: (\vc{x},\vc{y})\in\mathcal{D}_0\cap\mathcal{D}_1 \Big\},
\end{equation}
where
$\mathcal{D}_1=\{ (\vc{x},\vc{y})  \in \bS_{\C^{2n}}: \vc{x}^*S\vc{x}=\vc{y}^*S\vc{y}\}.$

From proposition \ref{bild_expr}, an element of the bild $B(A)$ has the form
\begin{eqnarray*}
  \omega &=& \vc{x}^*A\vc{x}+\vc{y}^* A^*\vc{y}, \\
   &=& \alpha^2\vc{x}_\bS^*A\vc{x}_\bS+(1-\alpha^2)\vc{y}_\bS^* A^*\vc{y}_\bS,
\end{eqnarray*}
where $(\vc{x},\vc{y})\in\bS_{\C^{2n}}$ and $\alpha^2=\|\vc{x}\|^2$. We conclude the following:
\begin{corollary}\label{cor_bild}
Let $A \in\M_n(\C)$, then
\[
B(A)\subset \mathrm{conv}\{W_{\C}(A), W_{\C}(A^*)\}.
\]
\end{corollary}
This result can also be obtained from the fact that $B(A)\subseteq W_{\C}(\chi_A)$ (\cite[Theorem 9.1]{Ki}) and that, for $A \in\M_n(\C)$,
\begin{equation}\label{chi_A_complexmatrix}
 W_{\C}(\chi_A)=\Conv \{W_{\C}(A), W_{\C}(A^*)\}.
\end{equation}


Corollary  \ref{cor_bild} leads us to the conclusion that the quaternionic and complex numerical radius of a $n\times n$ complex matrix $A$ coincide.

\begin{corollary}\label{numerical radius}
Let $A \in\M_n(\C)$. The quaternionic and the complex numerical radius of $A$ coincide.
\end{corollary}

\begin{proof}
It is enough to show that $w_{\Hq}(A)\leq w_{\C}(A)$. From corollary  \ref{cor_bild} we have
\[
\max\{|z|: z\in B(A)\} \leq \max\{|z|: z\in \mathrm{conv}\{W_{\C}(A), W_{\C}(A^*)\}\}.
\]
Given $z \in \mathrm{conv}\{W_{\C}(A), W_{\C}(A^*)\}$, there exists $z_1 \in W_{\C}(A)$ and $z_2 \in W_{\C}(A^*)$  such that
$z=\alpha z_1+(1-\alpha) z_2$, $\alpha\in [0,1]$. Then
$|z| \leq \max \{|z_1|, |z_2|\}$ and so $|z|\leq \max\{w_{\C}(A), w_{\C}(A^*) \}=w_{\C}(A)$.
Therefore, $w_{\Hq}(A)\leq  w_{\C}(A)$.
\end{proof}

The previous result also follows from \cite[Theorem 3.1]{K}, where it is stated that $w_{\Hq}(A)=w_{\C}(\chi_A)$. Since $A$ is a complex matrix, from (\ref{chi_A_complexmatrix}) we have $w_{\Hq}(A)=w_{\C}(A)$.

Therefore, when $A$ is a $n\times n$ complex matrix the numerical radius of $A$ is a norm. However, contrary to what it is stated in \cite[Proof of Theorem 3.3]{K}, when $A\in\M_n(\Hq)$, $\omega(A)$ is not a norm as the next example shows.

\begin{example}\label{tlvbk}
Let $A=\left[\begin{array}{cc} 1 & h\\ 0 & 1 \end{array}\right]$ with $h \in \bS_\Hq$.  This matrix can be written as the sum of a real diagonal $D$ and a nilpotent matrix. Thus by \cite[Theorem 4.2]{CDM2} we can conclude that $W(A)=\D_\Hq(1,1/2)$. And the numerical radius of $A$ is $\omega(A)=3/2$.

On the other hand let $ih=z \in \Hq$, so that $iA=\left[\begin{array}{cc} i & z\\ 0 & i \end{array}\right]$. Computing $\vc{x^*} iA \vc{x}$ for $\vc{x} \in \bS_{\Hq^2}$ we obtain
$
\vc{x^*} iA \vc{x}=x_1^* i x_1 + x_2^*i x_2+x_1^*zx_2,
$
 whose norm, using triangle's inequality , is
\[
\|\vc{x^*} iA \vc{x}\|\leq |x_1|^2+ |x_2|^2+|x_1||x_2| \leq 1+\max_{x^2+y^2=1} xy=3/2.
\]
In the previous equation we have equality if all the vectors are parallel (looking at quaternions as vectors in $\R^4$).  Thus we have $\|\vc{x^*} iA \vc{x}\|=3/2$ if and only if
\[
y_{1}^* i y_{1}= y_{2}^*i y_{2}=y_{1}^*z y_{2}, \quad \text{ with } y_{i}=x_{i, \bS}, \text{ for } i=1,2.
\]
Using the first equality we conclude that $\big(y_{2}y_{1}^*\big)i = i \big(y_{2}y_{1}^*\big)$.  Since $\big(y_{2}y_{1}^*\big)$ commutes with $i$, we have that $y_{2}y_{1}^*$ is complex. From the second equality we get  $z= y_{1} y_{2}^*i $, thus $z$ is also complex. That is, a necessary condition for $\|\vc{x^*} \big(iA\big) \vc{x}\|=3/2$ is that $z$ must be complex. Then, if $h$ is such that $z=ih$ is not complex,
\[
\omega (iA) =\max_{\vc{x} \in \bS_{\Hq^2}} \|\vc{x^*} iA \vc{x}\| <3/2.
\]
Then $\omega (iA) \neq |i| \omega (A)=\omega(A)$, and, in conflict with the complex case, the quaternionic  numerical radius is not a norm.

\end{example}


Next theorem gives the shape of the quaternionic numerical range for a general $n \times n$ complex matrix. This will depend on the complex  numerical range and two real values, equal to the largest and smallest real values in the quaternionic numerical range. Accordingly,  we define
\begin{align}
\underline{v}\equiv &\min B(A)\cap \R\label{defvmin}\\
\overline{v} \equiv &\max  B(A)\cap \R.\label{defvmax}
\end{align}

The following auxiliary result will be used in the proof of next theorem.
\begin{lemma}
Let $A=H+Si\in \M_n(\C)$. Let $\vc{z}_1, \vc{z}_2\in \bS_{\C^n}$ be such that $\vc{z}_1^* S \vc{z}_2=0$. Then
\[
[\omega_1, \omega_2]\subseteq B(A),
\]
where $\omega_1=\vc{z}_1^*A \vc{z}_1$ and $\omega_2=\vc{z}_2^*A^* \vc{z}_2$.
\end{lemma}

\begin{proof}
The result follows from proposition \ref{bild_expr} with $\vc{x}=\sqrt{\alpha}\vc{z}_1$ and $\vc{y}=\sqrt{1-\alpha}\vc{z}_2$, with $0\leq \alpha \leq 1$.
\end{proof}

\begin{theorem}\label{upperbild}
Let $A \in\M_n(\C)$. The upper bild of $A$ is given by
\[
B^+(A)=\mathrm{conv}\,\{W_{\C}^+(A),W_{\C}^+(A^*), \underline{v}, \overline{v}\}.
\]
\end{theorem}

\begin{proof}
We know that $W_{\C}(A) \subseteq W_{\Hq}(A)$. Hence,
$W_{\C}^+(A)\subseteq W_{\Hq}(A) \cap \C^+  =B^+(A).$

From $(W_{\C}(A))^*=W_{\C}(A^*)$ and $(W_\Hq(A))^*=W_\Hq(A)$, we have $W_{\C}^+(A^*)\subseteq B^+(A)$. Since the upper bild is convex,
$\mathrm{conv}\,\{ W_{\C}^+(A), W_{\C}^+(A^*)\} \subseteq B^+(A).$
 Therefore,
$\Conv\,\{W_{\C}^+(A),W_{\C}^+(A^*), \underline{v}, \overline{v}\}\subseteq B^+(A).$

We will now prove the converse inclusion. Let $w\in B^+(A)$. From Proposition \ref{bild_expr}, we know that, for some $(\vc{x}, \vc{y})\in \bS_{\C^{2n}}$,   $\vc{x}^*S\vc{y}=0$ and
\begin{align*}
\omega & =\vc{x}^*A\vc{x}+\vc{y}^* A^*\vc{y}.
\end{align*}
Let $\omega_1=\vc{x}_{\bS}^* A\vc{x}_{\bS}$ and $\omega_2=\vc{y}_{\bS}^* A^*\vc{y}_{\bS}$, then $\omega$ is a convex combination of $\omega_1$ and $\omega_2$, that is, $\omega=\alpha \omega_1+(1-\alpha)\omega_2$, $\alpha\in [0,1]$.
We will consider three cases.

If $\omega_1\in W_{\C}^+(A)$ and $\omega_2\in W_{\C}^+(A^*)$ then clearly
\[
\omega\in \Conv\,\{W_{\C}^+(A),W_{\C}^+(A^*), \underline{v}, \overline{v}\}.
\]

Let us now consider the case where $\omega_1\in W_{\C}^-(A)$ and $\omega_2\in W_{\C}^+(A^*)$. Let $r=[\omega_1,\omega_2]\cap\R$, with $\{\omega_1,\omega_2\}\nsubseteq \R$ ($\omega_1,\omega_2\in\R$ was treated in the first case).
Since $\omega\in [\omega_1,\omega_2]$ is an element of the upper bild, then $\omega\in [r, \omega_2]$. If $r\in [\underline{v}, \overline{v}]$, $\omega$ can be rewritten as convex combination of $\omega_2, \underline{v}, \overline{v}$. Therefore, $\omega\in \Conv\,\{W_{\C}^+(A),W_{\C}^+(A^*), \underline{v}, \overline{v}\}$. Thus, we only need to prove that, in fact, $r\in [\underline{v}, \overline{v}]$. From the previous lemma, we know that $r\in B(A)$ and we have $r\in [\underline{v}, \overline{v}]$. The last case, where $\omega_1\in W_{\C}^+(A)$ and $\omega_2\in W_{\C}^-(A^*)$, is similar.
\end{proof}

When $S$ is positive (semi) definite, then $W_{\C}(A) \subseteq \C^+$ and therefore $W_{\C}(A^*)\subseteq \C^-$. An immediate implication of the previous theorem is that, in this case, the upper bild is the  convex hull of $W_{\C}(A)$ and the two reals $\underline v, \overline v$.
\begin{corollary}
\label{upperbild_positive}
Let $A=H+Si \in \M_n(\C)$. The upper bild of $A$ is:
\begin{enumerate}[(i)]
  \item $B^+(A)=\mathrm{conv}\,\{W_{\C}(A), \underline{v}, \overline{v}\}$, if $S$ is positive (semi) definite;
  \item  $B^+(A)=\mathrm{conv}\,\{W_{\C}(A^*), \underline{v}, \overline{v}\}$, if $S$ is negative (semi) definite.
\end{enumerate}
\end{corollary}

Theorem \ref{upperbild} provides the shape of the upper bild in terms of complex numerical range and two real values $\underline{v}$ and $\overline{v}$. However these values might be difficult to calculate, as they involve a maximization over non trivial restrictions. Therefore to characterize cases where the calculation of the numerical range is simplified is important. Next corollary covers one of such cases. Specifically, it gives a necessary and sufficient condition for the equality of the bild and the complex numerical range. It becomes clear that when the complex numerical range of a complex matrix is symmetric regarding conjugation, then the bild coincides with the complex numerical range.
\begin{corollary}\label{bild_complex_nr}
Let $A \in\M_n(\C)$. Then $W_{\C}(A) = W_{\C}(A^*)$ if and only if $B(A)=W_{\C}(A)$. Moreover, $B(A)$ is convex.
\end{corollary}
\begin{proof}
Since $W_{\C}(A)\subseteq W_{\Hq}(A)$ then $W_{\C}(A)\subseteq B(A)$. By corollary \ref{cor_bild}, $B(A)\subseteq \Conv \{W_{\C}(A), W_{\C}(A^*)\}$. Since $W_{\C}(A)=W_{\C}(A^*)$, we have  $ B(A)=W_{\C}(A).$

On the other hand, since $B(A)=W_{\C}(A)$ and using that $\left(W_{\C}(A)\right)^*=W_{\C}(A^*)$, we have $(B(A))^*=\left(W_{\C}(A)\right)^*$. Therefore, $B(A)=W_{\C}(A^*)$ and so  $W_{\C}(A)=W_{\C}(A^*)$. Together with Toeplitz-Hausdorff theorem we conclude that $B(A)$  is convex.
\end{proof}

Notice that,  when $B(A)$ is convex does not imply that $W_{\C}(A) = W_{\C}(A^*)$, as example \ref{example_intro} shows.

 Another immediate consequence of the previous result is that the bild of a real matrix $A \in \M_n(\R)$ is the complex numerical range of $A$, \emph{i.e} $B(A)=W_{\C}(A)$. Thus corollary  \ref{bild_complex_nr} encompasses \cite[theorem 3.7]{CDM}.

We will see now two examples where we can describe the shape of the bild of $A$, and therefore of $W_\Hq(A)$. We would like to point out that, using our results, there are several complex matrices for which the quaternionic numerical range can be computed from the known results of the complex numerical range.

\begin{example}\label{ex_matrix_p}
Let $A\in \M_3(\C)$ be the matrix
\begin{equation*}
A=\left(
                 \begin{array}{ccc}
                  i & 0 & 1 \\
                   0 & i & 0 \\
                    0 & 0 & -i\\
                \end{array}
              \right).
\end{equation*}
 From \cite[Theorem 2.4]{KRS}, $W_\C(A)$ is an ellipse with foci at $\lambda_1=i$, $\lambda_2=-i$ and minor axis equal to 1. In this case, since $W_\C(A^*)=(W_\C(A))^*$, we have
 $W_\C(A^*)$  is also the ellipse described above.  So  $W_\C(A)=W_\C(A^*)$ and
from corollary \ref{bild_complex_nr} we conclude that $B(A)=W_\C(A)$ is an ellipse and of course convex.\squareend
\end{example}

\begin{example}\label{ex_vmin_vmax}
Consider the complex matrix
\begin{equation*}
A=\left(
                 \begin{array}{ccc}
                  i & 0 & 1 \\
                   0 & i & 0 \\
                    0 & 0 & i\\
                \end{array}
              \right)\in M_3(\C).
\end{equation*}
From \cite[Theorem 4.1]{KRS}, $W_\C(A)$ is the disk with center $i$ and radius $\frac{1}{2}$. Therefore,  $W_\C(A^*)$ is the disk with center $-i$ and radius $\frac{1}{2}$. It follows that $W_\C^+(A)=\mathbb{D}(i,\frac{1}{2})$ and $W_\C^+(A^*)=\emptyset$.

Now, given $X=(x,y,z)\in\bS_{\Hq^3}$, we have
\[
X^*AX=x^*ix+y^*iy+z^*iz+x^*z.
\]
Denote the real part of $X^*AX$ by
\[
Re(X^*AX)=f(x,y,z)=x_0z_0+x_1z_1+x_2z_2+x_3z_3
\]
and the imaginary part by
\[
Im(X^*AX)=F_1(x,y,z)i+F_2(x,y,z)j+F_3(x,y,z)k.
\]
Then, $\underline{v}$ (resp., $\overline{v}$) is the minimum (resp., maximum) of the function $f(x,y,z)$, subject to the constrains
\[
F_1(x,y,z)=F_2(x,y,z)=F_3(x,y,z)=0\,\,\,\textrm{and}\,\,\,|x|^2+|y|^2+|z|^2=1.
\]
Using a MATLAB program for optimization with constrains we find that $\underline{v}=-\frac{1}{4}$ and $\overline{v}=\frac{1}{4}$. Invoking theorem \ref{upperbild}, we conclude that the upper bild of $A$ is
\[
B^+(A)=\mathrm{conv}\Bigg\{\mathbb{D}(i,\tfrac{1}{2}), -\frac{1}{4}, \frac{1}{4}\Bigg\}.
\]\squareend
\end{example}

Theorem \ref{upperbild} generalizes theorem $6.1$ of the outstanding \emph{tour de force} \cite{ST}. In the next proposition it is established that for $2\times2$ complex matrices the upper bild is either $B^+(A)=\Conv\{W_{\C}^+(A), W_{\C}^+(A^*)\}$, $B^+(A)=\mathrm{conv}\,\{W_{\C}^+(A), v\}$, for some $v \in \R$, or $B^+(A)= W_{\C}^+(A)$. We thus provide a different and more straightforward proof for the above mentioned theorem of \cite{ST} (see the appendix for the proof).

\begin{corollary}\label{upperbild_2times2}
Let $A=H+Si \in\M_2(\C)$. The upper bild of $A$ is:
\begin{enumerate}[(i)]
  \item $B^+(A)=\mathrm{conv}\,\{W_{\C}^+(A), W_{\C}^+(A^*)\}$, if $S$ is indefinite;
  \item $B^+(A)=\mathrm{conv}\,\{W_{\C}^+(A), v\}$, for some real value $v$, if $S$ is positive definite;
  \item $B^+(A)= W_{\C}^+(A)$, if $S$ is positive semi-definite.
\end{enumerate}
\end{corollary}

A natural question to ask is if the previous corollary, or part of it, can be generalized for $n>2$. That is,
  when the matrix $S$ is positive definite, positive semi-definite or indefinite, can the shape of the quaternionic numerical range given by theorem \ref{upperbild} be further simplified? Next example shows that it is not the case for $3 \times 3$ normal complex matrices.
\begin{example}
Consider the complex matrix $A=H+Si \in \M_3(\C)$, where $A=\diag (1+i,1+i,-i)$. Hence $W_{\C}(A)=[1+i,-i]$ and $\Conv\{W_{\C}^+(A), W_{\C}^+(A^*)\}=\Conv\{i, 1/2, 1+i\}$. Since matrix $A$ is normal, its upper bild is $\Conv\{i, 1/2, 1, 1+i\}$. We have, then, an example of a matrix $A$, with an indefinite matrix $S$, $S=diag (1,1,-1)$, where $B^+(A)\neq \Conv\{W_{\C}^+(A), W_{\C}^+(A^*)\}$.

On the other hand, by unitary similarity (in the quaternions) we have that the numerical range of $\tilde{A}=\diag (1+i,1+i,i)$ is equal to the numerical range of $A$. The complex numerical range of $\tilde{A}$ is $[i,1+i]$. Thus we have a matrix with positive definite $S$, where $B^+(A)\neq \Conv\{W_{\C}^+(A), v\}$. \squareend
\end{example}
%
%




\section{Appendix}

We now prove corollary \ref{upperbild_2times2}. To do so it is  important to understand for which vectors $\vc{q}=\vc{x}+\vc{y}j \in \bS_{\Hq^2}$ the element $\vc{q}^*A\vc{q}$ is real and what is the real part of the numerical range. Next preparatory lemma provides conditions for the first of these matters.
\begin{lemma}
Let $A=H+Si \in\M_2(\C)$, where $S=\diag (\lambda_1,\lambda_2)$ and $\lambda_1,\lambda_2 \in \R\backslash \{0\}$. Let  $\vc{q}=\vc{x}+\vc{y}j \in \bS_{\Hq^2}$ and $\vc{x}=(x_1,x_2), \vc{y}=(y_1,y_2)  \in \C^2$. If $\vc{q}^*A\vc{q} \in \R$ then
\begin{equation}\label{2by2_lemma1}
\lambda_1x_1^*x_2 |y_1|^2 =-\lambda_2y_2y_1^*|x_2|^2.
\end{equation}
and
\begin{equation} \label{2by2_lemma2}
\big( \lambda_1 |y_1|^2+\lambda_2|y_2|^2 \big)\big(\lambda_2|x_2|^2 - \lambda_1|y_1|^2\big)=0.
\end{equation}
\end{lemma}
\begin{proof}

 From (\ref{Real part of bild}), an element of the numerical range is real if and only if  $\vc{x}^*S\vc{y}=0$ and $\vc{x}^* S\vc{x}- \vc{y}^*S\vc{y}=0$. We have
 \begin{equation}\label{2by2_thm_eq0}
  \vc{x}^*S\vc{y}=0 \Leftrightarrow \lambda_1x_1^* y_1+\lambda_2x_2^*y_2=0
 \end{equation}
 and
 \begin{equation}\label{2by2_thm_eq1}
   \vc{x}^* S\vc{x}- \vc{y}^*S\vc{y}=0 \Leftrightarrow  \lambda_1 |x_1|^2+\lambda_2 |x_2|^2-\lambda_1 |y_1|^2 -\lambda_2 |y_2|^2=0.
 \end{equation}
If $x_2=0$ and $y_1=0$, (\ref{2by2_lemma1}) and (\ref{2by2_lemma2})  follow trivially. If $x_2=0$ and $y_1\neq 0$, from (\ref{2by2_thm_eq0}), $x_1=0$ and from (\ref{2by2_thm_eq1}) we have $\lambda_1|y_1|^2+\lambda_2|y_2|^2=0$. Then  (\ref{2by2_lemma1}) and (\ref{2by2_lemma2})  follow. If $x_2\neq 0$ and $y_1= 0$, an analogous reasoning proves (\ref{2by2_lemma1}) and (\ref{2by2_lemma2}). It remains to see the case $x_2\neq 0$ and $y_1\neq 0$.

From (\ref{2by2_thm_eq0}) we have
\begin{align}
& \lambda_1x_1^*=-\lambda_2x_2^*y_2\frac{y_1^*}{|y_1|^2} \label{2by2_lemma3}\\
\Leftrightarrow& \frac{\lambda_1x_1^*x_2 }{|x_2|^2} =-\frac{\lambda_2y_2y_1^*}{|y_1|^2} \nonumber
\end{align}
and (\ref{2by2_lemma1}) follows. The previous equality implies that
\begin{align}\label{2by2_lemma4}
&\lambda_1^2|x_1|^2=\lambda_2^2\frac{|x_2|^2|y_2|^2}{|y_1|^2}.
\end{align}
On the other hand, from (\ref{2by2_thm_eq1}), we have
\begin{align}
& \frac{\lambda_2^2}{\lambda_1}\frac{|x_2|^2|y_2|^2}{|y_1|^2}+\lambda_2 |x_2|^2-\lambda_1 |y_1|^2 -\lambda_2 |y_2|^2=0 \nonumber\\
 \Leftrightarrow & \lambda_2^2|x_2|^2|y_2|^2 +\lambda_1\lambda_2 |x_2|^2|y_1|^2-\lambda_1^2 |y_1|^2|y_1|^2 -\lambda_1\lambda_2 |y_2|^2|y_1|^2=0 \nonumber\\
    \Leftrightarrow & \lambda_2|y_2|^2 (\lambda_2|x_2|^2 -\lambda_1|y_1|^2) + \lambda_1 |y_1|^2 \big( \lambda_2 |x_2|^2-\lambda_1 |y_1|^2| \big)=0 \nonumber\\
       \Leftrightarrow & \big( \lambda_1 |y_1|^2+\lambda_2|y_2|^2 \big)\big(\lambda_2|x_2|^2 - \lambda_1|y_1|^2\big)=0. \label{2by2_thm_eq2}
\end{align}
This concludes the proof.
\end{proof}

For $2\times 2$ complex matrices, with $\lambda_1\neq 0$ or $\lambda_2\neq 0$,  the upper-bild of $A$ can be characterized as follows.
\begin{repcorollary}{upperbild_2times2}
Let $A=H+Si \in\M_2(\C)$. The upper bild of $A$ is:
\begin{enumerate}[(i)]
  \item $B^+(A)=\mathrm{conv}\,\{W_{\C}^+(A), W_{\C}^+(A^*)\}$, if $S$ is indefinite;
  \item $B^+(A)=\mathrm{conv}\,\{W_{\C}^+(A), v\}$, for some real value $v$, if $S$ is positive definite;
  \item $B^+(A)= W_{\C}^+(A)$, if $S$ is positive semi-definite.
\end{enumerate}
\end{repcorollary}

\begin{proof}
Let $\vc{q}=\vc{x}+\vc{y}j\in\bS_{\Hq^2}$, with $(\vc{x},\vc{y})\in\bS_{\C^{4}}$ and define
\[
\mathcal{R}=\{(\vc{x},\vc{y})\in\bS_{\C^{4}}:\vc{q}^*A\vc{q}\in\R\}
\]

For $(\vc{x},\vc{y})\in\mathcal{R}$, from the previous lemma, we have
\begin{equation}\label{2by2_lemma2_prop}
  \lambda_1 |y_1|^2+\lambda_2|y_2|^2 =0 \quad \text{or} \quad \lambda_2|x_2|^2 - \lambda_1|y_1|^2 =0.
\end{equation}

\textbf{Case (i):}

When $S=\mathrm{diag}(\lambda_1,\lambda_2)$ is indefinite let, without loss of generality, $\lambda_1>0$ and $\lambda_2<0$. In view of theorem \ref{upperbild}, we need to prove that $\underline{v},\overline{v}\in W_{\C}(A)$.


If $\lambda_1|y_1|^2=-\lambda_2|y_2|^2$ then equation (\ref{2by2_thm_eq1}) implies that $\lambda_1|x_1|^2+\lambda_2|x_2|^2=0$. In this case we have that $\vc{x}^*S\vc{x}=\vc{y}^*S\vc{y}=0$.

If $\lambda_2|x_2|^2-\lambda_1|y_1|^2=0$ then $|x_2|=|y_1|=0$. From equation (\ref{2by2_thm_eq1}), $\lambda_1|x_1|^2-\lambda_2|y_2|^2=0$ and so $|x_1|=|y_2|=0$. But this contradicts the fact that $(\vc{x},\vc{y})\in\bS_{\C^{4}}$, so this case is ruled out.

Therefore, from  (\ref{num_range_element2}),  $\overline{v}\in W_{\Hq}(A)\cap\R$ verifies
\begin{align*}
\overline{v} &= \vc{x}^* A\vc{x} + \vc{y}^* A^*\vc{y}\\
&= \alpha^2\vc{x}_\bS^* A\vc{x}_\bS + (1-\alpha^2)\vc{y}_\bS^* A^*\vc{y}_\bS\\
&=\alpha^2\vc{x}_\bS^* H\vc{x}_\bS + (1-\alpha^2)\vc{y}_\bS^* H\vc{y}_\bS,
\end{align*}
where $\alpha^2=\|\vc{x}\|^2=1-\|\vc{y}\|^2\in[0,1]$, that is, $\overline{v}$ is a convex combination of $\vc{x}_\bS^* H\vc{x}_\bS$ and $\vc{y}_\bS^* H\vc{y}_\bS$. Thus,
\[
\overline{v} \leq \mathrm{max}\,\{\vc{x}_\bS^* H\vc{x}_\bS, \vc{y}_\bS^* H\vc{y}_\bS\}.
\]
We have that $\vc{x}_{\bS}^*H\vc{x}_{\bS}=\vc{x}_{\bS}^*A\vc{x}_{\bS}$ and $\vc{y}_{\bS}^*H\vc{y}_{\bS}=\vc{y}_{\bS}^*A\vc{y}_{\bS}$.
It follows that
\[
\overline{v} \leq \mathrm{max}\,\{\vc{x}_\bS^* A\vc{x}_\bS, \vc{y}_\bS^* A\vc{y}_\bS\}\leq \mathrm{max} W_{\C}(A)\cap \R.
\]
Since $W_{\C}(A)\subseteq W_{\Hq}(A)$,
\[
\mathrm{max}W_{\C}(A)\cap \R \leq \mathrm{max}W_{\Hq}(A)\cap \R =\overline{v}.
\]
Hence $\overline{v}\in W_{\C}(A)$.

A similar reasoning allows us to prove that $\underline{v}\in W_{\C}(A)$.

\textbf{Case (ii):}

Suppose $S=\mathrm{diag}(\lambda_1,\lambda_2)$ is positive definite, i.e. $\lambda_1>0$ and $\lambda_2>0$. We will prove that $\underline{v}=\overline{v}=v$, that is, $(\vc{x},\vc{y})\mapsto\vc{x}^*H\vc{x}+\vc{y}^*H\vc{y}$ is constant over the elements $(\vc{x},\vc{y})\in\mathcal{R}$.

For $(\vc{x},\vc{y})\in\mathcal{R}$, equation (\ref{2by2_lemma2_prop}) holds if $|y_1|=|y_2|=0$ or $\lambda_1|y_1|^2-\lambda_2|x_2|^2=0$. The first case implies, from (\ref{2by2_thm_eq1}), that $|x_1|=|x_2|=0$, which contradicts $(\vc{x},\vc{y})$ to be an element of $\bS_{\C^4}$. The latter case,
$\lambda_1|y_1|^2=\lambda_2|x_2|^2$, implies, from (\ref{2by2_lemma1}), that $x_1^*x_2=-y_1^*y_2$. Replacing on the equation $\vc{x}^*\vc{x}+\vc{y}^*\vc{y}=1$ we get
\[
|x_1|^2+|y_1|^2=\frac{\lambda_2}{\lambda_1+\lambda_2}, \text{ and } |x_2|^2+|y_2|^2=\frac{\lambda_1}{\lambda_1+\lambda_2}.
\]
It follows that
\begin{align*}
\vc{x}^*H\vc{x}+\vc{y}^*H\vc{y}=&h_{11}(|y_1|^2+|x_1|^2)+h_{22}(|y_2|^2+|x_2|^2)\\
=&h_{11}\frac{\lambda_2}{\lambda_1+\lambda_2}+h_{22}\frac{\lambda_1}{\lambda_1+\lambda_2}.
\end{align*}
Therefore, $\vc{x}^*H\vc{x}+\vc{y}^*H\vc{y}$ only depends on the entries of $H$ and $S$, being constant over $\mathcal{R}$.


\textbf{Case (iii):}

When $S=\mathrm{diag}(\lambda_1,\lambda_2)$ is positive semi-definite let, without loss of generality,  $\lambda_1>0$ and $\lambda_2=0$.
For $\omega\in B(A)$, from (\ref{B(A) in terms of H,S}),
\begin{equation*}
\omega=\vc{x}^*H\vc{x}+\vc{y}^* H \vc{y}+(\vc{x}^*S\vc{x}-\vc{y}^*S\vc{y})i, \quad (\vc{x},\vc{y})\in\mathcal{D}_0.
\end{equation*}

Since $(\vc{x}, \vc{y})\in \mathcal{D}_0$ and $\vc{x}^*S\vc{x}$= $\vc{y}^*S\vc{y}$, we have $|x_1|=|y_1|=0.$

For $H=\begin{pmatrix}
    h_{11} & h_{12}\\
    h_{21} & h_{22}\\
  \end{pmatrix}
$
we have
\begin{eqnarray*}
\omega     &=& \begin{pmatrix}
   0 & x_2^*\\
  \end{pmatrix}\begin{pmatrix}
    h_{11} & h_{12}\\
    h_{21} & h_{22}\\
  \end{pmatrix}\begin{pmatrix}
   0 \\
   x_2\\
  \end{pmatrix}+ \begin{pmatrix}
   0 & y_2^*\\
  \end{pmatrix}\begin{pmatrix}
    h_{11} & h_{12}\\
    h_{21} & h_{22}\\
  \end{pmatrix}\begin{pmatrix}
   0 \\
   y_2\\
  \end{pmatrix}\\
     &=& |x_2|^2 h_{22}+|y_2|^2 h_{22} \\
     &=&  h_{22},
  \end{eqnarray*}
since $(\vc{x},\vc{y})\in \bS_{\C^{4}}$.

Therefore, we can write $\omega= \begin{pmatrix}
   0 & 1\\
  \end{pmatrix}A\begin{pmatrix}
   0 \\
   1\\
  \end{pmatrix}\in W_\C^+(A)$.
\end{proof}
%



\begin{thebibliography}{99}

\bibitem[CDM1]{CDM} L. Carvalho, C. Diogo, S. Mendes, \emph{A bridge between quaternionic and complex numerical ranges}, Linear Algebra and
its Applications, \textbf{581} (2019), 496--504.

\bibitem[CDM2]{CDM2} L. Carvalho, C. Diogo, S. Mendes, \emph{On the convexity and circularity
of the numerical range of
nilpotent quaternionic matrices}, New York J. Math. \textbf{25} (2019), 1385--1404.



\bibitem[Ki]{Ki} R. Kippenhahn, O\emph{n the numerical range of a matrix}, Translated from the German by Paul F. Zachlin and Michiel E. Hochstenbach. Linear Multilinear Algebra \textbf{56:1-2 }(2008), 185-225.

\bibitem[KRS]{KRS} D. Keeler, L. Rodman, I. Spitkovsky, \emph{The numerical range of $3\times 3$ matrices}, Linear Algebra and
its Applications, \textbf{252} (1997), 115--139.

\bibitem[K]{K} P. Kumar, \emph{A note on convexity of sections of quaternionic numerical range}, Linear Algebra and its Applications, \textbf{572} (2019), 92--116.


\bibitem[R]{R}   L. Rodman, \emph{Topics in Quaternion Linear Algebra}, Princeton University Press, 2014.

\bibitem[ST]{ST} W. So, R. C. Thompson, \emph{Convexity of the upper complex plane part of the numerical range of a quaternionic matrix}, Linear and Multilinear Algebra \textbf{41} (1996): 303-365.

\bibitem[STZ]{STZ} W. So, R. C. Thompson, F. Zhang, \emph{The numerical range of normal matrices with quaternion entries}, Linear and Multilinear Algebra, \textbf{37} (1994), 175--195.


\bibitem[Ye1]{Ye1} Y. H. Au-Yeung,\emph{\emph{ On the convexity of the numerical range in quaternionic Hilbert space}}, Linear and
Multilinear Algebra, \textbf{16} (1984), 93--100.

\bibitem[Ye2]{Ye2} Y. H. Au-Yeung, \emph{A short proof of a theorem on the numerical range of a normal
quaternionic matrix}, Linear and Multilinear Algebra, \textbf{39:3} (1995), 279--284.

\bibitem[Zh]{Zh} F. Zhang, \emph{Quaternions and matrices of quaternions}, Linear Algebra and its Applications, \textbf{251} (1997), 21--57.

\end{thebibliography}
\end{document}